\documentclass[12pt,a4paper]{article}
\usepackage{amsmath, amsfonts,
   amssymb, amsbsy,
   amsthm, amscd, euscript}

\newtheorem{theor}{Theorem}
\theoremstyle{definition}
\newtheorem{state}[theor]{Proposition}
\newtheorem{lemma}[theor]{Lemma}
\newtheorem{define}{Definition}
\newtheorem{example}{Example}
\theoremstyle{remark}
\newtheorem{rem}{Remark}

\newcommand{\cEv}{\partial}
\newcommand{\BBR}{\mathbb{R}}

\newcommand{\cE}{\mathcal{E}}

\newcommand{\cF}{\mathcal{F}}
\newcommand{\cH}{\mathcal{H}}

\newcommand{\cX}{{\EuScript X}}    
\newcommand{\cY}{{\EuScript Y}}    
\newcommand{\boldb}{{\boldsymbol{b}}}

\newcommand{\bn}{{\boldsymbol{n}}}
\newcommand{\bp}{{\boldsymbol{p}}}
\newcommand{\bq}{{\boldsymbol{q}}}
\newcommand{\br}{{\boldsymbol{r}}}
\newcommand{\bs}{{\boldsymbol{s}}}
\newcommand{\bu}{{\boldsymbol{u}}}

\newcommand{\bH}{{\boldsymbol{H}}}
\newcommand{\bL}{{\boldsymbol{L}}}

\newcommand{\gothg}{\mathfrak{g}}
\newcommand{\gA}{\mathfrak{A}}

\newcommand{\vph}{\varphi}
\newcommand{\dd}{\partial}
\newcommand{\Id}{{\mathrm d}}

\DeclareMathOperator{\id}{id}
\DeclareMathOperator{\sym}{sym}

\DeclareMathOperator{\Hom}{Hom}

\DeclareMathOperator{\CDiff}{\mathcal{C}Diff}

\newcommand{\lshad}{[\![}
\newcommand{\rshad}{]\!]}

\newcommand{\by}[1]{\textit{{#1}}}
\newcommand{\jour}[1]{\textit{{#1}}}
\newcommand{\vol}[1]{\textbf{{#1}}}
\newcommand{\book}[1]{\textrm{{#1}}}

\newcommand{\ib}[3]{ \{\!\{ {#1},{#2} \}\!\}_{{#3}} }

\title
{Variational Lie algebroids\\ and homological evolutionary vector fields}

\date{October 31, 2010; accepted December 29, 2010}

\author
{Arthemy V. Kiselev,
\thanks{{}\,
Mathematical Institute, University of Utrecht, P.O.Box 80.010, 3508 TA Utrecht, The Netherlands.}
\thanks{{}\,
\textit{Address for correspondence}: 
Johann Bernoulli Institute for Mathematics and Computer Science,
University of Groningen, 
P.O.Box~407, 9700~AK Groningen, The Netherlands.
\textit{E-mail}: \texttt{A.V.Kiselev\symbol{"40}rug.nl}.}
\quad%
Johan W. van de Leur\,${}^*$\thanks{\textit{E-mail}: \texttt{J.W.vandeLeur\symbol{"40}uu.nl}.}}

\begin{document}
\maketitle

\begin{abstract}\noindent%
We define Lie algebroids over infinite jet spaces
and establish their equivalent representation through
homological evolutionary vector fields.

\end{abstract}

\paragraph*
{Introduction.}
The construction of Lie algebroids over smooth manifolds is important in differential geometry (particularly, in Poisson geometry) and appears in various models of mathematical physics (e.g., in Poisson sigma\/-\/models).
We extend the classical definition of Lie algebroids over smooth manifolds \cite{Vaintrob} to the construction of variational Lie algebroids over infinite jet spaces. We define these structures in a standard way via vector bundles and also 
through homological vector fields $Q^2=0$ on infinite jet super\/-\/bundles, then proving the equivalence. Our generalization of the classical construction manifestly respects the geometry which appears under mappings between smooth manifolds. For this reason, the variational picture, which we develop here, more fully grasps the geometry of strings in space\/-\/time \cite{AKZS}.

First, we very briefly recall 
the definition of classical Lie algebroids; 
we refer to \cite{Vaintrob} 
for more details (see also \cite{Voronov} or
the surveys \cite{
YKSDB
} and references therein).
The standard examples of Lie algebroids over usual manifolds are the tangent bundle or the Poisson algebroid structure of the cotangent bundle to a Poisson manifold.
At the same time,
Lie algebras are toy examples of Lie algebroids over a point. 
Likewise, 
we view smooth manifolds $M^m$ as (very poor) `infinite jet spaces' for mappings of the point into them, so that the manifolds become `fibres' in the `bundles' $\pi\colon M^m\to\{\text{pt}\}$. 
This very informal (but equally productive\,!) understanding is our starting point.
We shall show what becomes of the Lie algebroids over $M^m$ 
under the extension of the base point to an $n$-\/dimensional 
manifold $\varSigma^n$, which yields a true jet 
bundle $J^\infty(\varSigma^n\to M^m)$.

Our results are structured in this paper as follows.
It is readily seen that a literal transfer of the classical
definition, see p. \pageref{DefLieAlgd}, 
over smooth manifolds
to the infinite jet bundles is impossible because the Leibniz 
rule \eqref{LeibnizInAlg} is lost \textit{ab initio} (to the same extent as it is lost for the commutators of evolutionary vector fields or for the variational Poisson bracket \cite{ClassSym}). For this reason, 
in section \ref{SecDef} we take, as 
the new definition, an appropriate consequence of the classical one: namely, the commutation closure for the images of the anchors. 
(Remarkably, this 
is often \textsl{postulated} [for convenience, rather than derived] 
as a part of the definition of a Lie algebroid, e.g., 
see \cite{Vaintrob,Voronov} vs \cite{Herz,YKSMagri}.)
In these terms, our main result of \cite{SymToda} states that
the non\/-\/periodic 2D Toda equations associated with semi\/-\/simple complex Lie algebras \cite{Leznov}, being the representatives for the vast class of hyperbolic Euler\/--\/Lagrange systems of Liouville type \cite{SokolovUMN}, are variational Lie algebroids.

\begin{rem}
We duly recall 
that Lie algebroids over the spaces of finite jets of sections for the \textsl{tangent} bundle $\pi\colon T\varSigma\to \varSigma^n$ were defined in \cite{Kumpera}. However, in this paper we let the vector bundle $\pi$ be arbitrary. 
Instead of the tangent bundle (as in \cite{Kumpera}), we use model examples of a principally different nature. Namely, our illustrations come from the geometry of differential equations: e.g., we address Hamiltonian (non-)\/evolutionary systems \cite{Topical} or the 2D Toda chains \cite{Leznov}, 
the solutions of which are $r$ functions in two variables for every semi\/-\/simple complex Lie algebra of rank $r$. In other words, we analyse the general case when the base and fibre dimensions in the bundle $\pi$ are not related.
\end{rem}

In section \ref{SecVarQ}, we represent the structure of a variational Lie algebroid $\gA$, defined in the above sense, by the homological evolutionary vector field $Q$ on the infinite jet bundle to $\Pi\gA$ viewed as the vector bundle with parity reversed fibres. This is again nontrivial due to the use of the Leibniz rule (now missing) in the equivalence proof for the classical definition. 
We also impose the natural nondegeneracy assumption on the variational anchors so that the odd field $Q$ is well defined.
To 
illustrate our assertion, we let the variational anchors be
Hamiltonian differential operators and obtain the fields $Q$ which are
themselves Hamiltonian with respect to the corresponding variational Poisson bi\/-\/vectors and the canonical symplectic structure of \cite{KuperCotangent}.

In the sequel, the ground field is $\BBR$ and all mappings are supposed to be $C^\infty$-\/smooth. Likewise, we suppose that all total
differential operators are local (i.e., polynomial in total derivatives).
We assume
that all spaces are purely even, 
i.e., that the manifolds are not superized, unless we state the opposite
explicitly when using the parity reversion $\Pi$. 
We employ the standard techniques from the geometry of differential equations \cite{ClassSym,Olver}; for convenience, we summarize the necessary notions in the beginning of section \ref{SecDef}. 
Because all our reasonings are local,
we can conveniently restrict to local charts 
and, instead of the jet bundles for mappings $\varSigma^n\to M^m$ of manifolds, we consider the jet spaces for $m$\/-\/dimensional vector bundles $\pi\colon E^{m+n}\to\varSigma^n$.
By default, we define all structures 
on the empty infinite jet spaces, identifying
autonomous evolution equations with evolutionary vector fields, which
is standard.
The delicate interplay between the definitions of all structures at hand on the empty jet spaces and on differential equations would require a much longer
description; this will be the object of another paper. 
In the meantime, 
we refer to the review~\cite{Topical}. 

We follow the notation of \cite{SymToda,TMPhGallipoli,d3Bous}, 
borrowing it in part from \cite{YKSDB,
ClassSym,
Olver
}.
This paper is a continuation of the papers \cite{SymToda,d3Bous}, 
to which we refer for motivating examples of variational Lie algebroids.
Some of the results, which we present here in 
detail, were briefly reported in the preprint \cite{Twente}.

\paragraph*
{Classical Lie algebroids: a review.}\label{SecAlgd}
\noindent%
Let $M^m$ be a smooth real $m$-\/dimensional manifold 
($1\leq m\leq+\infty$)
and $TM\to M^m$ be its tangent bundle.

\begin{define}[\cite{Vaintrob}]\label{DefLieAlgd}
A \textsl{Lie algebroid} over a 
manifold $M^m$ is a vector 
bundle $\xi\colon\Omega^{d+m}\to M^m$ 
the space of sections $\Gamma\Omega$ of which
is equipped with a Lie algebra
structure $[\,,\,]_A$ together with a morphism of the 
bundles $A\colon\Omega\to TM$,
called the \textsl{anchor}, such that the Leibniz rule
\begin{equation}\label{LeibnizInAlg}
[f\cdot\cX,\cY]_A=f\cdot[\cX,\cY]_A
  -\bigl(A(\cY)f\bigr)\cdot\cX
\end{equation}
holds for any
$\cX,\cY\in\Gamma\Omega$ and any $f\in C^\infty(M^m)$. 
\end{define}

Essentially, 
the anchor in a Lie algebroid is a specific 
fiberwise\/-\/linear mapping from a given vector bundle 
over a smooth manifold $M^m$ to its tangent bundle.

\begin{lemma}[\cite{Herz}
]\label{PropMorphismFollows}
The anchor $A$ maps the bracket $[\,,\,]_A$ for sections
of the vector bundle $\xi$ 
to the Lie bracket $[\,,\,]$ for sections of the tangent bundle
to the manifold $M^m$.
\end{lemma}

This property is a consequence
of the Leibniz rule \eqref{LeibnizInAlg} and the Jacobi identity
for the Lie algebra structure $[\,,\,]_A$ in $\Gamma\Omega$.

\begin{lemma}[\cite{Vaintrob}]\label{RemHVF}
Equivalently, a Lie algebroid structure on $\Omega$ is
a homological vector field $Q$ on $\Pi\Omega$
(take the fibres of $\Omega$, reverse their parities, and thus obtain
the new super\/-\/bundle $\Pi\Omega$ over $M^m$).
These homological vector fields, 
which are differentials on
$C^\infty(\Pi\Omega)=\Gamma(\bigwedge^\bullet\Omega^*)$, equal
\begin{equation}\label{HomVF}
Q=A_i^\alpha(u)b^i\frac{\dd}{\dd u^\alpha}
 -\frac{1}{2}b^i\,c^k_{ij}(u)\,b^j\frac{\dd}{\dd b^k},
\qquad [Q,Q]=0\ \Leftrightarrow\ 2Q^2=0,
\end{equation}
where
\begin{itemize}
\item
$(u^\alpha)$ is a system of local coordinates near a point $u\in M^m$,
\item
$(p^i)$ are local coordinates along the $d$-\/dimensional fibres of $\Omega$
and $(b^i)$ are the respective coordinates on $\Pi\Omega$, and
\item
$[e_i,e_j]_A=c^k_{ij}(u)e_k$ give
the structural constants for a $d$-\/element
local basis $(e_i)$ of sections
in $\Gamma\Omega$ over the point~$u$ and 
$A(e_i)=A_i^\alpha(u)\cdot\dd/\dd u^\alpha$ is the image of $e_i$
under the anchor $A$. 
\end{itemize}
\end{lemma}


\begin{proof}[Sketch of the proof]
The anti\/-\/commutator $[Q,Q]=2Q^2$ of the odd vector field $Q$ with itself is a vector field. Its coefficient of $\dd/\dd u^\alpha$ vanishes because $A$ is the Lie algebra homomorphism by 
Lemma \ref{PropMorphismFollows}. 
The equality to zero of the coefficient of $\dd/\dd b^q$ is achieved in three steps. First, we notice that the application of the second term in \eqref{HomVF} to itself by the graded Leibniz rule for vector fields yields the overall numeric factor $\tfrac{1}{4}$, but it doubles due to the skew\/-\/symmetry of the structural constants $c^k_{ij}$. Second, we recognize the right\/-\/hand side of the
Leibniz rule \eqref{LeibnizInAlg}, with $c^k_{ij}(u)$ for $f\in C^\infty(M^m)$, in the term
\begin{equation}\label{BeforeTriple}
\frac{1}{2}\sum\nolimits_{ijn}b^ib^jb^n\left(
-A_n^\alpha\,\frac{\dd}{\dd u^\alpha}\bigl(c^q_{ij}(u)\bigr)
 +c^\ell_{ij}c^q_{\ell n}
\right)\cdot\frac{\dd}{\dd b^q}.
\end{equation}
Third, we note that the \textsl{cyclic} permutation $i\to j\to n\to i$ of the odd variables $b^i$,\ $b^j$, $b^n$ does not change the sign of \eqref{BeforeTriple}. We thus triple it by taking the sum over the permutations (and duly divide by three) and, as the coefficient of $\dd/\dd b^q$, obtain the Jacobi identity for $[\,,\,]_A$,
\begin{equation}\label{JacobiUsual}
\tfrac{1}{2}\cdot\tfrac{1}{3}\cdot\sum\nolimits_{\circlearrowright}
\bigl[[e_i,e_j]_A, e_n\bigr]_A=0.
\end{equation}
The zero in its right\/-\/hand side calculates the required coefficient.
Thus, $Q^2=0$. 
\end{proof}


\section{Variational Lie algebroids}\label{SecDef}
\noindent%
For consistency, we first summarize the notation.
Let $\varSigma^n$ be an $n$-\/dimensional orientable smooth real manifold,
and let $\pi\colon E^{m+n}\xrightarrow[M^m]{}\varSigma^n$ be a vector 
bundle over it 
with $m$-\/dimensional fibres $M^m\ni u=(u^1,\ldots,u^m)$.
By $J^\infty(\pi)$ we denote the infinite jet space over 
$\pi$, and 
we set $\pi_\infty\colon J^\infty(\pi)\to \varSigma^n$. We denote by $u_\sigma$,
$|\sigma|\geq0$, its fibre coordinates. Then $[u]$ stands for
the differential dependence on $u$ and its derivatives up to some finite order,
and we put $\cF(\pi)\mathrel{{:}{=}}C^\infty(J^\infty(\pi))$, 
understanding it in the standard way 
(as the inductive limit of filtered algebras, see \cite{ClassSym}).

Let $\xi\colon N\to \varSigma^n$ 
be another vector 
bundle over the same base $\varSigma^n$ and take its pull\/-\/back $\pi_\infty^*(\xi)\colon N 
\mathbin{{\times}_{\varSigma^n}}J^\infty(\pi)\to J^\infty(\pi)$ along $\pi_\infty$.
By definition, the $\cF(\pi)$-\/module of sections
$\Gamma\bigl(\pi_\infty^*(\xi)\bigr)=\Gamma(\xi)\mathbin{{\otimes}_{C^\infty(\varSigma^n)}}\cF(\pi)$ is called \textsl{horizontal}, see \cite{Lstar}.

\begin{example}\label{CanonicalInduced}
Three horizontal $\cF(\pi)$-\/modules are canonically associated with
every jet space $J^\infty(\pi)$, see, e.g., \cite{Lstar}. 
First, let $\xi\mathrel{{:}{=}}\pi$, 
whence we obtain the $\cF(\pi)$-\/module $\Gamma(\pi_\infty^*(\pi))
=\Gamma(\pi)\mathbin{{\otimes}_{C^\infty(\varSigma^n)}}\cF(\pi)$.
The shorthand notation 
is $\varkappa(\pi)\equiv\Gamma(\pi_\infty^*(\pi))$.
Its sections $\vph\in\varkappa(\pi)$ 
are in a one\/-\/to\/-\/one correspondence with
the $\pi$-\/vertical evolutionary derivations 
$\cEv_\vph=\sum_\sigma \frac{\Id^{|\sigma|}}{\Id x^\sigma}
   (\vph)\cdot\dd/\dd u_\sigma$ 
on $J^\infty(\pi)$, here
$\tfrac{\Id}{\Id x}=\tfrac{\dd}{\dd x}+u_x\,\tfrac{\dd}{\dd u}+\cdots$
is the total derivative.
For all $\psi$ such that $\cEv_\vph(\psi)$ makes sense,
the linearizations $\ell_\psi^{(u)}$ are defined by
$\ell_\psi^{(u)}(\vph)=\cEv_\vph(\psi)$, where $\vph\in\varkappa(\pi)$.

Second, let $\xi$ be the $q$th exterior power of the cotangent bundle $T^*\varSigma$, here $q\leq n$. Taking the pull\/-\/back $\pi_\infty^*(\xi)$, we obtain the horizontal $\cF(\pi)$-\/module $\bar{\Lambda}^q(\pi)$ of its sections, which are the \textsl{horizontal $q$-\/forms} on the jet space $J^\infty(\pi)$. In local coordinates, they are written as $h\bigl(x,[u]\bigr)\cdot\Id x^{i_1}\wedge\dots\wedge\Id x^{i_q}$, where $h\in\cF(\pi)$ and $1\leq i_1<\dots<i_q\leq n$.

Thirdly, take the 
module $\bar{\Lambda}^n(\pi)$ 
of highest $\pi$-\/horizontal forms on $J^\infty(\pi)$.
Then we denote by $\widehat{\varkappa}(\pi)=
\Hom_{\cF(\pi)}\bigl(\varkappa(\pi),\bar{\Lambda}^n(\pi)\bigr)$ the horizontal $\cF(\pi)$-\/module dual to $\varkappa(\pi)$.
\end{example}




\begin{rem}\label{RemComposite}   
The structure of fibres in the bundle $\xi$ can be composite 
and their dimension infinite. With the following canonical example
we formalize 
the 2D Toda geometry \cite{SymToda}. 
Namely, let $\zeta\colon I^{r+n}\xrightarrow[W^r]{}\varSigma^n$ be a 
vector 
bundle with $r$-\/dimensional fibres in which $w=(w^1$,\ $\ldots$,\ $w^r)$ are local coordinates. Consider the infinite jet bundle $\xi_\infty\colon J^\infty(\xi)\to \varSigma^n$ and, by definition, set either $\xi=\zeta_\infty\circ\zeta_\infty^*(\zeta)\colon\varkappa(\xi)\to \varSigma^n$ or
$\xi=\zeta_\infty\circ\zeta_\infty^*(\widehat{\zeta})\colon\widehat{\varkappa}(\xi)\to \varSigma^n$. 

To avoid an inflation 
of formulas, we shall always briefly denote 
by $\Gamma\Omega\bigl(\xi_\pi\bigr)$
the horizontal $\cF(\pi)$-\/module at hand.
\end{rem}

By definition, mappings between 
$\cF(\pi)$-\/modules are called \textsl{total differential operators} if they are sums of compositions of $\cF(\pi)$-\/linear maps and liftings of vector fields from the base $\varSigma^n$ onto $J^\infty(\pi)$ by the Cartan connection $\dd/\dd x\mapsto\Id/\Id x$.
The main objects of our study are total differential operators 
(i.e., matrix, linear differential operators in total derivatives) 
that take values 
in the Lie algebra $\gothg(\pi)=\bigl(\varkappa(\pi),[\,,\,]\bigr)$. 

\begin{define}\label{DefFrob}
Let the above assumptions and notation hold.
Consider a total differential operator
$A\colon\Gamma\Omega\bigl(\xi_\pi\bigr)\to\varkappa(\pi)$
the 
image of which is closed under the commutation 
in $\gothg(\pi)$: 
\begin{equation}\label{EqDefFrob}  
[\text{im}\,A, \text{im}\,A]\subseteq\text{im}\,A
\ 
\Longleftrightarrow
\ 
\bigl[A(\bp),A(\bq)\bigr]=A\bigl([\bp,\bq]_A\bigr),\quad
\bp,\bq\in\Gamma\Omega\bigl(\xi_\pi\bigr).
\end{equation}
The operator $A$ transfers the Lie algebra structure 
$[\,,\,]\bigr|_{\mathrm{im}\,A}$ in a Lie subalgebra of $\gothg(\pi)$
to the 
bracket $[\,,\,]_A$
in the quotient $\Gamma\Omega\bigl(\xi_\pi\bigr)/\ker A$.
%
%
Then the triad 
\begin{equation}\label{EqVarLieAlgd}
\bigl(\Gamma\Omega\bigl(\xi_\pi\bigr),[\,,\,]_A\bigr)
\xrightarrow{\ A\ {}}\bigr(\varkappa(\pi),[\,,\,]\bigr)
\end{equation}
is the \textsl{variational Lie algebroid} $\gA$ over the infinite jet space $J^\infty(\pi)$,
and the Lie algebra homomorphism $A$ is the \textsl{variational anchor}.
\end{define}

Essentially, 
the variational anchor in a variational Lie algebroid is a specific linear mapping 
from a given horizontal module of sections of an induced bundle over
the infinite jet space $J^\infty(\pi)$
to the prescribed horizontal module of generating sections for evolutionary derivations on $J^\infty(\pi)$.

\begin{example}\label{ExKdV}
Consider the Hamiltonian operator 
$A_2=-\tfrac{1}{2}\,\tfrac{\Id^3}{\Id x^3}+
w\,\tfrac{\Id}{\Id x}+\tfrac{\Id}{\Id x}\circ w$,
which yields the second Poisson structure for the 
Korteweg\/--\/de Vries equation $w_t=-\tfrac{1}{2}w_{xxx}+3ww_x$. 
The image of $A_2$ is closed under commutation:
the bracket $[\,,\,]_{A_2}$ 
is \cite{SokolovUMN}
\begin{equation}\label{BrA2}
[p,q]_{A_2}=\cEv_{A_2(p)}(q)-\cEv_{A_2(q)}(p)
  +
{\tfrac{\Id}{\Id x}(p)\cdot q -p\cdot\tfrac{\Id}{\Id x}(q)}.
\end{equation}
\end{example}

\begin{example}\label{Exd3Bous}
In \cite{d3Bous} we demonstrated that the dispersionless $3$-\/component
Boussinesq system of hydrodynamic type admits a two\/-\/parametric family of
finite deformations $[\,,\,]_{\boldsymbol{\epsilon}}$ for the
standard bracket $[\,,\,]$ of its symmetries $\sym\cE$. 
For this, we used 
two 
recursion differential operators 
$R_i\colon\sym\cE\to\sym\cE$, $i=1,2$,
the images of which are closed under commutation 
and which are \textsl{compatible}
in this sense, spanning the two\/-\/dimensional space of the 
operators $R_{\boldsymbol{\epsilon}}$ with involutive images.
We obtain the new brackets $[\,,\,]_{\boldsymbol{\epsilon}}$  
via \eqref{EqDefFrob} (c.f. \cite{
Contractions} and references therein): 
$\bigl[R_{\boldsymbol{\epsilon}}(p),R_{\boldsymbol{\epsilon}}(q)\bigr]=
R_{\boldsymbol{\epsilon}}\bigl([p,q]_{\boldsymbol{\epsilon}}\bigr)$ for $p,q\in\sym\cE$ and $\boldsymbol{\epsilon}\in\BBR^2\setminus\{0\}$.
\end{example}

\begin{rem}
We standardly regard autonomous evolution equations as 
the $\pi$-\/vertical evolutionary vector fields on
the infinite jet bundles $J^\infty(\pi)$ for 
vector bundles $\pi$ over $\varSigma^n\ni x$.
We thus reduce the construction at hand to the 
``absolute'' case of empty jet spaces.

Generally, there is no Frobenius theorem for the involutive distributions
spanned by the images of the variational anchors.
Nevertheless, the Euler\/--\/Lagrange systems $\cE$ 
of Liouville type provide an exception \cite{SymToda,SokolovUMN} because of their outstanding symmetry structure: 
the variational anchors 
yield the involutive distributions
of evolutionary vector fields that are tangent to the \textsl{integral
manifolds}, namely, to the differential equations $\cE$ themselves.
In these terms, 
for a \textsl{given} differential equation $\cE\subset J^\infty(\pi)$, we accordingly restrict the definition of variational Lie algebroid
to the infinite\/-\/dimensional manifold $\cE$ such that the variational
anchors yield its 
symmetries, and, if necessary, we shrink the fibres of $\xi$.
\end{rem}

\begin{example}
The generators of the Lie algebra of point symmetries for the
$(2+1)$-\/dimensional
`heavenly' Toda equation $u_{xy}=\exp(-u_{zz})$, see \cite{BoyerFinley},
are 
$\vph^x=\widehat{\square}^x\bigl(p(x)\bigr)$ or
$\vph^y=\widehat{\square}^y\bigl(\bar{p}(y)\bigr)$, where
$p$ and $\bar{p}$ are arbitrary smooth functions and 
$\widehat{\square}^x=u_x+\tfrac{1}{2}z^2\,\tfrac{\Id}{\Id x}$, 
$\widehat{\square}^y=u_y+\tfrac{1}{2}z^2\,\tfrac{\Id}{\Id y}$.
The image of each operator is 
closed under commutation such that
$
[p,q]_{\widehat{\square}^x}=\cEv_{\widehat{\square}^x(p)}(q)
  -\cEv_{\widehat{\square}^x(q)}(p)
+p\cdot\tfrac{\Id}{\Id x}(q) - \tfrac{\Id}{\Id x}(p)\cdot q
$ 
for any $p(x)$ and $q(x)$, and similarly for $\widehat{\square}^y$.
\end{example}

\begin{example}\label{ExVarCoTan}
Two distinguished constructions of the variational Lie algebroids emerge from the variational (co)\/tangent bundles over $J^\infty(\pi)$, see Example \ref{CanonicalInduced} and Remark \ref{RemComposite}. 
These two cases involve an important intermediate component,
which we regarded as the Miura substitution in \cite{SymToda}.

Namely, consider the induced fibre bundle $\pi_\infty^*(\zeta)$ and fix its section $w$.
Making no confusion, we continue denoting by the same letter $w$
the fibre coordinates in $\zeta$ and the fixed 
section $w[u]\in\Gamma\bigl(\pi_\infty^*(\zeta)\bigr)$,
which is a nonlinear differential operator in $u$.

Obviously, the substitution $w=w[u]$ converts 
the horizontal $\cF(\zeta)$-\/modules 
to the submodules of horizontal $\cF(\pi)$-\/modules.\footnote{%
For instance, 
we have $\zeta=\pi$ for the recursion
operators $\varkappa(\pi)\to\varkappa(\pi)$, 
see Example \ref{Exd3Bous}. 
It is standard to identify the domains of Hamiltonian operators with 
$\widehat{\varkappa}(\pi)$, see Example \ref{ExKdV}. 
Here we set $w=\id\colon\Gamma(\pi)\to\Gamma(\zeta)$ in both cases;
this is why the auxiliary bundle $\zeta$ becomes ``invisible''
in the standard exposition (e.g.\ in~\cite{Lstar,JKGolovko2008}).}
By this argument, we obtain the modules
\[
\varkappa(\zeta){\bigr|}_{w[u]\colon J^\infty(\pi)\to\Gamma(\zeta)}\quad
\text{and}\quad
\widehat{\varkappa}(\zeta){\bigr|}_{w[u]\colon J^\infty(\pi)\to\Gamma(\zeta)},
\]
where the latter is the module of sections of the pull\/-\/back by $\pi_\infty^*$ for the $\bar{\Lambda}^n(\zeta)$-\/dual to the induced bundle $\zeta_\infty^*(\zeta)$.
We emphasize that by this approach we preserve the correct
transformation rules for the sections 
in $\varkappa(\zeta)$ or $\widehat{\varkappa}(\zeta)$
under (unrelated\,!) reparametrizations of the fibre coordinates $w$ and $u$ in the bundles $\zeta$ and $\pi$, respectively. 

We say that the linear operators $A\colon\varkappa(\zeta){\bigr|}_{w[u]}\to
\varkappa(\pi)$ and $A\colon\widehat{\varkappa}(\zeta){\bigr|}_{w[u]}\to\varkappa(\pi)$ subject to \eqref{EqDefFrob} are the variational anchors of \textsl{first} and \textsl{second kind}, respectively.
Under 
differential reparametrizations 
$\tilde{u}=\tilde{u}[u]\colon
J^\infty(\pi)\to\Gamma(\pi)$ and $\tilde{w}=\tilde{w}[w]\colon
J^\infty(\zeta)\to\Gamma(\zeta)$, 
the operators $A$ of first kind are transformed according to the formula
$A
\mapsto\tilde{A}=\ell_{\tilde{u}}^{(u)}\circ A\circ\ell_w^{(\tilde{w})}
  \Bigr|_{\substack{w=w[u]\\u=u[\tilde{u}]}}$. 
Respectively, the operators of second kind obey the rule 
$A
\mapsto\tilde{A}=\ell_{\tilde{u}}^{(u)}\circ A\circ
 \bigl(\ell_{\tilde{w}}^{(w)}\bigr)^\dagger 
  \Bigr|_{\substack{w=w[u]\\u=u[\tilde{u}]}}$, 
where the symbol $\dagger$ denotes the adjoint operator.
The recursion operators with involutive images, which we addressed 
in \cite{d3Bous}, are examples of the anchors of first kind.
All Hamiltonian operators and, more generally, Noether operators
with involutive images are variational anchors of second 
kind \cite{d3Bous}.
The operators $\square$,\ $\overline{\square}$ 
that yield symmetries of 
the Euler\/--\/Lagrange Liouville\/-\/type systems
are also anchors of second kind \cite{SymToda,TMPhGallipoli};
they are `non\/-\/skew\/-\/adjoint generalizations
of Hamiltonian operators' \cite{SokolovUMN} in exactly this sense.
\end{example}


In the remaining part of this section we 
analyse the standard structure of the induced bracket $[\,,\,]_A$ on the quotient of $\Gamma\Omega\bigl(\xi_\pi\bigr)$ by $\ker A$.
By the Leibniz rule, two sets of summands appear in the bracket of
evolutionary vector fields $A(\bp),A(\bq)$ that belong to the image of the variational anchor $A$:  
\begin{equation}\label{LeibnizForOp}
\bigl[A(\bp),A(\bq)\bigr]=A\bigl(\cEv_{A(\bp)}(\bq)-\cEv_{A(\bq)}(\bp)\bigr)
  +\bigl(\cEv_{A(\bp)}(A)(\bq)-\cEv_{A(\bq)}(A)(\bp)\bigr).
\end{equation}
In the first summand we have used the permutability of
evolutionary derivations and total derivatives. 
The second summand hits the image of $A$ by construction.
Consequently, the 
Lie algebra structure $[\,,\,]_A$ on the domain of $A$
equals
\begin{equation}\label{EqOplusBKoszul} 
[\bp,\bq]_A=\cEv_{A(\bp)}(\bq)-\cEv_{A(\bq)}(\bp)+ \{\!\{\bp,\bq\}\!\}_A.
\end{equation}
Thus, the bracket $[\,,\,]_A$, which is defined up to $\ker A$,
contains the two standard summands and
the bi\/-\/differential skew\/-\/symmetric part 
$\{\!\{\,,\,\}\!\}_A\in\CDiff\bigl(\Gamma\Omega(\xi_\pi)\times\Gamma\Omega(\xi_\pi)
\to\Gamma\Omega(\xi_\pi)\bigr)$. 
For any $\bp,\bq,\br 
\in\Gamma\Omega(\xi_\pi)$,
the Jacobi identity for \eqref{EqOplusBKoszul} is 
\begin{align}
0&=\sum\nolimits_\circlearrowright \bigl[ [\bp,\bq]_A,\br\bigr]_A =
 \sum\nolimits_\circlearrowright
  \bigl[ \cEv_{A(\bp)}(\bq)-\cEv_{A(\bq)}(\bp)
  +\ib{\bp}{\bq}{A}, \br\bigr]_A \notag\\
%
%
{}&=\sum\nolimits_\circlearrowright\Bigl\{
\cEv_{A\bigl(\cEv_{A(\bp)}(\bq)-\cEv_{A(\bq)}(\bp)\bigr)}(\br)
-\cEv_{A(\br)}\bigl(\cEv_{A(\bp)}(\bq)-\cEv_{A(\bq)}(\bp)\bigr)
  \notag \\
{}&\qquad{}
+\ib{\cEv_{A(\bp)}(\bq)-\cEv_{A(\bq)}(\bp)}{\br}{A}\notag\\
{}&\qquad{}
 +\cEv_{A\bigl(\ib{\bp}{\bq}{A}\bigr)}(\br)
 -\underline{\cEv_{A(\br)}\bigl(\ib{\bp}{\bq}{A}\bigr)}
 +\ib{\ib{\bp}{\bq}{A}}{\br}{A}\Bigr\}. \label{JacobiKoszul}
\end{align}
The underlined summand contains derivations of the coefficients of
$\ib{\,}{\,}{A}$, which belong to $\cF(\pi)$. 
Even if the action of evolutionary fields 
$\cEv_\vph^{(u)}=\vph\tfrac{\dd}{\dd u}+\ldots$ on 
the the arguments of $A$ 
is set to zero (which makes sense, see below),
these summands may not vanish. 
Note that the Jacobi identity for $[\,,\,]_A$ then amounts to 
\begin{equation}\tag{\ref{JacobiKoszul}${}'$}\label{JacobiKoszulShort}
\sum\nolimits_{\circlearrowright}\Bigl(
-\cEv^{(\bu)}_{A(\br)}\bigl(\ib{\bp}{\bq}{A}\bigr)
+\ib{\ib{\bp}{\bq}{A}}{\br}{A}\Bigr)=0.
\end{equation}
Formulas \eqref{EqDefFrob} and \eqref{EqOplusBKoszul}, 
formula \eqref{JacobiKoszul}, and \eqref{JacobiKoszulShort} will play the same role in the proof of Theorem \ref{ThEvHomVF} in the next section as Lemma \ref{PropMorphismFollows} and, respectively, formulas \eqref{JacobiUsual} and \eqref{LeibnizInAlg} [or \eqref{BeforeTriple}] did for the equivalence of the two classical definitions.


\section{Homological evolutionary vector fields}\label{SecVarQ}\noindent%
In this section we represent variational Lie algebroids using the homological evolutionary vector fields $Q$ on the infinite jet super\/-\/bundles that we naturally associate with Definition \ref{DefFrob}. 
To achieve this goal, we make two preliminary steps. Namely, we first identify the odd \textsl{neighbour}\footnote{By Yu. I. Manin's definition, the \textsl{neighbours} of a Lie algebra $\gothg$ are $\gothg^*$,\ $\Pi\gothg$, and $\Pi\gothg^*$, where $\Pi$ is the parity reversion functor,
c.\,f. \cite{Voronov}.%
} 
$\overline{J^\infty}(\Pi\xi_\pi)$ of the infinite horizontal jet bundle $\overline{J^\infty}(\xi_\pi)$ over $J^\infty(\pi)$. Second, we impose the conditions upon the variational anchors $A$ such that the resulting homological field $Q$ becomes well defined
on $\overline{J^\infty}(\Pi\xi_\pi)$; 
here we use the method which we already applied in \cite{SymToda}.


We recall that, by construction, the bundle $\xi_\pi\colon\Omega(\xi_\pi)\to J^\infty(\pi)$ is induced by $\pi_\infty$ from the vector 
bundle $\xi$ over the base $\varSigma^n$. 
Let us consider the horizontal infinite jet bundle $\overline{J^\infty}(\xi_\pi)$ for the vector bundle $\pi_\infty^*(\xi)$ over $J^\infty(\pi)$, 
see \cite{Topical,Lstar,JKGolovko2008} for more details and examples.
Finally let us reverse the parity of the fibres in the bundle $\overline{J^\infty}(\xi_\pi)\to J^\infty(\pi)$, 
we thus denote by $\Pi\colon\bp\mapsto\boldb$
the parity reversion; right below we recall what the variables
$\bp$ and $\boldb$~are. 
This yields the horizontal infinite jet \textsl{super}\/-\/bundle, which we denote by $\overline{J^\infty}(\Pi\xi_\pi)$, over $J^\infty(\pi)$.
We now summarize the notation:\label{PList}
\begin{itemize}
\item $x$ is the $n$-\/tuple of local coordinate(s) on the base manifold $\varSigma^n$;
\item $u_{\boldsymbol{\sigma}}$, where $|\boldsymbol{\sigma}|\geq0$, are the jet coordinates in a fibre of the bundle $\pi_\infty\colon J^\infty(\pi)\to \varSigma^n$;
\item $\bp$ is the fibre coordinate in 
the induced 
bundle $\pi_\infty^*(\xi)$ over $J^\infty(\pi)$, 
and 
$\bp$,\ $\bp_x$,\ $\ldots$,\ $\bp_{\boldsymbol{\tau}}$
with any multi\/-\/index $\boldsymbol{\tau}$
are the fibre coordinates in the horizontal infinite jet bundle 
$\overline{J^\infty}(\xi_\pi)$ over it;
\item $\boldb$,\ $\boldb_x$,\ $\ldots$,\ $\boldb_{\boldsymbol{\tau}}$, 
where $|\boldsymbol{\tau}|\geq0$,
are the odd fibre coordinates in the horizontal infinite jet super\/-\/bundle $\overline{J^\infty}(\Pi\xi_\pi)\to J^\infty(\pi)$;
\item the variational anchor $A$, being a 
linear differential operator in total derivatives, 
is a fiberwise linear function on 
either $\overline{J^\infty}(\xi_\pi)$ or $\overline{J^\infty}(\Pi\xi_\pi)$ and takes values in $\varkappa(\pi)$, which is the space of generating
sections for evolutionary vector fields on $J^\infty(\pi)$;
\item $[\,,\,]_A$ is the bracket \eqref{EqOplusBKoszul} on $\Gamma\Omega(\xi_\pi)$.
It is defined up to the kernel of $A$; 
we now take \textsl{any} representative of 
the bi\/-\/differential operation $\ib{\,}{\,}{A}$ from
the equivalence class and tautologically extend it to the
bi\/-\/linear function on the horizontal jet space.
\end{itemize}
Over each point $\theta^\infty\in J^\infty(\pi)$, the odd fibre of the horizontal infinite jet super\/-\/bundle $\overline{J^\infty}(\Pi\xi_\pi)$
contains the linear subspace $\ker A$ that is determined by the (infinite prolongation of the) linear equation $A{\bigr|}_{\theta^\infty}(\boldb)=0$.
It is important that the linear space of its solutions does not depend on the point $[\boldb]$ in the fibre, because the differential operator $A{\bigr|}_{\theta^\infty}$ depends only on the base point $\theta^\infty\in J^\infty(\pi)$ and the equation is linear. To eliminate the functional freedom in the kernel of $A$, 
we require that the operator $A$ be \textsl{nondegenerate} in the sense 
of \cite{SymToda}: $A\circ\nabla=0$ implies $\nabla=0$.


\begin{theor}\label{ThEvHomVF}
Let $A$ be a nondegenerate variational anchor and all above assumptions and notation hold.
Whatever be the choice of the bracket $\ib{\,}{\,}{A}\mod\ker A$,
the following odd evolutionary vector field on the horizontal infinite jet super\/-\/space $\overline{J^\infty}(\Pi\xi_\pi)$ is homological\textup{:}
\begin{equation}\label{VarHomVF}
Q=\cEv^{(\bu)}_{A(\boldb)}
 -\tfrac{1}{2}\cEv^{(\boldb)}_{\ib{\bp}{\bq}{A} {\bigr|}_{
\substack{
  \bp\mathrel{{:}{=}}\boldb\\ 
  \bq\mathrel{{:}{=}}\boldb
} 
} 
}, 
\qquad [Q,Q]=2Q^2=0.
\end{equation}
Moreover, either there is a unique canonical representative
in the equivalence class $Q\mod\cEv^{(\boldb)}_{\bn}$
of such fields in the fibre 
of $\overline{J^\infty}(\Pi\xi_\pi)\to J^\infty(\pi)$
over each point $\theta^\infty\in J^\infty(\pi)$,
here $\bn\in\ker A$ so that $\cEv^{(\boldb)}_{\bn}$ is 
a symmetry of the kernel, 
or, when the anchor $A$ is close to degenerate in the sense below, the uniqueness is achieved by fixing $\ib{\,}{\,}{A}$ on \textup{(}at most\textup{)} the boundary of a star\/-\/shaped domain centered at the origin of the fibre over $\theta^\infty$ in some finite\/-\/order horizontal jet space $\overline{J^{|\boldsymbol{\tau}|}}(\Pi\xi_\pi)$, $|\boldsymbol{\tau}|<\infty$.
\end{theor}

\begin{example}
The second Hamiltonian operator $A_2$ for the KdV equation (see above)
determines the homological vector field 
$Q=\cEv_{A_2(b)}^{(u)}+\cEv^{(b)}_{b b_x}$.
This is the simplest example that involves a nontrivial differential polynomial of degree two in the anti\/-\/commuting variable(s) $\boldb$ and its derivatives. 
\end{example}

\begin{proof}[Proof of Theorem \protect{\textup{\ref{ThEvHomVF}}}]
The anticommutator $[Q,Q]=2Q^2$ of the odd evolutionary vector field is an evolutionary vector field itself. Hence it suffices to find only the coefficients of $\dd/\dd\bu$ and $\dd/\dd\boldb$ in it; by construction, all other terms are the respective total derivatives of these two.

We first consider the odd velocity of $\bu$; it equals
\begin{align*}
\cEv^{(\bu)}_{A(\boldb)}(A)(\boldb)&-\tfrac{1}{2}A\bigl(
\ib{\bp}{\bq}{A}\bigr) {\bigr|}_{
\substack{
  \bp\mathrel{{:}{=}}\boldb\\ 
  \bq\mathrel{{:}{=}}\boldb
} 
}. 
\\ 
\intertext{Since $\boldb$ is odd, we double the minuend as follows:}
{}={}&\frac{1}{2}\Bigl(
\cEv^{(\bu)}_{A(\bp)}(A)(\bq)
-\cEv^{(\bu)}_{A(\bq)}(A)(\bp)
-A\bigl(\ib{\bp}{\bq}{A}\bigr)
\Bigr) {\Bigr|}_{
\substack{
  \bp\mathrel{{:}{=}}\boldb\\ 
  \bq\mathrel{{:}{=}}\boldb
} 
}. 
\end{align*}
By the definition of $\ib{\,}{\,}{A}$ in (\ref{LeibnizForOp}--\ref{EqOplusBKoszul}), the expression in parentheses equals zero.

Second, the coefficient of $\dd/\dd\boldb$ is equal to
\begin{align*}
-\tfrac{1}{2}&\cEv^{(\bu)}_{A(\boldb)}\bigl(\ib{\bp}{\bq}{A}
\bigr) {\bigr|}_{
\substack{
  \bp\mathrel{{:}{=}}\boldb\\ 
  \bq\mathrel{{:}{=}}\boldb
} 
} 
+\tfrac{1}{4}\cEv^{(\boldb)}_{\ib{\bp}{\bq}{A} {\bigr|}_{
\substack{
  \bp\mathrel{{:}{=}}\boldb\\ 
  \bq\mathrel{{:}{=}}\boldb
} 
} 
} 
\bigl(\ib{\br}{\bs}{A}\bigr) {\bigr|}_{
\substack{
  \br\mathrel{{:}{=}}\boldb\\ 
  \bs\mathrel{{:}{=}}\boldb
} 
}.\\ 
\intertext{The bracket $\ib{\,}{\,}{A}$ is skew\/-\/symmetric (in the usual sense);
therefore, the substitution of an odd $\boldb$ for both arguments doubles it. 
Besides, the odd derivation $\cEv^{(\boldb)}_{\ib{\bp}{\bq}{A}}$,
with $\boldb$ for $\bp$ and $\bq$,
acts on its argument by the graded Leibniz rule. By construction, it replaces every instance of a derivative of $\boldb$ with the respective derivative of $\ib{\boldb}{\boldb}{A}$. This yields}
{}={}&\frac{1}{2}\Bigl( -\cEv^{(\bu)}_{A(\boldb)}\bigl(\ib{\bp}{\bq}{A}\bigr)
+\ib{\ib{\bp}{\bq}{A}}{\br}{A} \Bigr) {\Bigr|}_{
\substack{
  \bp\mathrel{{:}{=}}\boldb\\ 
  \bq\mathrel{{:}{=}}\boldb\\
  \br\mathrel{{:}{=}}\boldb
} 
}. 
\end{align*}
We now triple this expression by taking the sum over the cyclic permutations
$\bp\to\bq\to\br\to\bp$, which does not contribute with any sign. The result
is one\/-\/sixth of the Jacobi identity \eqref{JacobiKoszulShort}. This proves our first claim.

\smallskip
We have established that $Q^2=0$ for any choice of the bracket
$\ib{\,}{\,}{A}+\bn([\boldb])$, where $\bn\in\ker A$ is arbitrary.
On one hand, the linear subspace $\ker A$ in the horizontal jet fibre over a point $\theta^\infty\in J^\infty(\pi)$ does not depend on a point $[\boldb]$ of that fibre. On the other hand, the representative $\bn\in\ker A$ can be nonzero and, in principle, depend on a point of the fibre. We claim that $\bn$ can be either trivialized or normalized in a well defined way uniformly on the entire fibre. 

First, by the appropriate shift along the kernel, let us normalize the bracket
$\ib{\,}{\,}{A}$ by the condition $\ib{[0]}{[0]}{A}\equiv0$ 
at the origin $[\boldb]=0$ of the fibre over $\theta^\infty\in J^\infty(\pi)$ 
in the vector bundle $\overline{J^\infty}(\Pi\xi_\pi)\to J^\infty(\pi)$.
Now we notice that the remaining addend $\bn$, which, at all fibre points, belongs to 
$\ker A$ that is common for them, is quadratic homogeneous in the horizontal jet coordinates $[\boldb]$: 
we have $\bn=\bn\bigl([\boldb]\cdot[\boldb]\bigr)$. 
The nondegeneracy assumption for the operator $A$ excludes the possibility of a free functional dependence on $[\boldb]\cdot[\boldb]$ for sections which belong to the domain of a nonzero operator $\nabla$ such that $A\circ\nabla\equiv0$.

Obviously, the differential order of the bracket $\ib{\,}{\,}{A}$ 
is finite and does not exceed the sum of differential orders of the operator $A$ and of its coefficients, which may depend on $[\bu]$. This implies that
$\bn=\bn\bigl((\boldb,\ldots,\boldb_{\boldsymbol{\tau}})\cdot
(\boldb,\ldots,\boldb_{\boldsymbol{\tau}})\bigr)$ 
with $|\boldsymbol{\tau}|<\infty$.
For the odd $\boldb$,\ $\ldots$,\ $\boldb_{\boldsymbol{\tau}}$, take
all pairwise products $\psi={}^{\textrm{t}}\bigl(
b^i\cdot b^j$,\ $\ldots$,\ $b_{\boldsymbol{\tau}}^i\cdot b_{\boldsymbol{\tau}}^j
\bigr)$ that are not identically zero,
and construct the $m$-\/tuple $\bn=\nabla(\psi)$ using the undetermined coefficients: let $\nabla^\alpha_\beta\in\BBR$ and $\bn^\alpha=\nabla^\alpha_\beta\cdot\psi^\beta$. The hypothesis $\bn\neq0$ formally does not contradict our initial assumption of $A$ being nondegenerate,
because the section $\psi$ thus obtained is not arbitrary (there may be fewer free parameters than the number of components in it due to the presence of polynomial relations between the components).
If there are no such nontrivial $\nabla$, the proof is over.
However, suppose there is a solution $\nabla\neq0$.
Then, under the rescaling by $\lambda\in\BBR$ in the fibre, the deviation
$\bn(\lambda)\in\ker A$ from the normalization $\bn\equiv0$ at the origin
grows quadratically in $\lambda$. 
Therefore, to finally fix the bracket $\ib{\,}{\,}{A}$,
it is sufficient to assign the values to $\bn$ at the points
on the boundary of a central\/-\/symmetric star\/-\/shaped domain, centered around the origin in the finite\/-\/dimensional fibre over $\theta^\infty$ 
of the finite 
jet space $\overline{J^{|\boldsymbol{\tau}|}}(\Pi\xi_\pi)$.
By this argument, we normalize $\ib{\,}{\,}{A}$ on the entire fibre
of $\overline{J^\infty}(\Pi\xi_\pi)$.
This defines the canonical representative $Q$ in the class \eqref{VarHomVF} 
on this fibre.
The proof is complete.
\end{proof}

\begin{rem}
In the beginning of the proof, we used the postulated property of the variational anchor $A$ to be the Lie algebra homomorphism. The next part of the proof holds due to the skew\/-\/symmetry of $\ib{\,}{\,}{A}$, the universal Leibniz rule for derivations, and the Jacobi identity for $[\,,\,]_A$.
In a sense, it is possible to regard the summands in \eqref{JacobiKoszulShort} as the right\/-\/hand side of a \textsl{new} `Leibniz rule' for the bracket \eqref{EqOplusBKoszul} in the variational Lie algebroid.
\end{rem}

\begin{rem}
The homological evolutionary vector field \eqref{VarHomVF} 
is quadratic in the odd variable $\boldb$ or its derivatives,
and contains no free term. Suppose, however, that this degree is arbitrary positive (possibly, infinite) and all other assumptions still hold.
This yields the Schlessinger\/--\/Stasheff 
homotopy deformation \cite{Stasheff}
of the Lie algebra structure in $\bigl(\Gamma\Omega(\xi_\pi),[\,,\,]_A\bigr)$. Namely, the Jacobi identity $[Q,Q]=0$ involves now not only the binary bracket $[\,,\,]_A$, but also ternary, quadruple, and other $N$-\/ary 
skew\/-\/symmetric brackets with the numbers of arguments $N\geq3$, 
see \cite{AKZS,Voronov,Stasheff,ForKac} and references therein. 

The non\/-\/existence of the variational $N$-\/vectors at $N\geq3$ for all evolutionary systems with invertible symbols\footnote{It is readily seen that
this assumption depends on the choice of local coordinates.
We conjecture that the same assertion holds for higher\/-\/order evolutionary systems with the nondegenerate linearizations of their right\/-\/hand sides.}
of differential order greater than one 
was argued in \cite{JKVerb3Vect}, see references therein. This now states the rigidity, in the homotopy sense, of the variational Poisson algebroid structures determined by Hamiltonian operators for such equations. 
However, 
the variational Poisson algebroids for equations that can not be 
cast to the above form (e.g., for gauge systems) may admit homotopy deformations.
\end{rem}


\paragraph*
{Variational master equation: an illustration.}
Let us study in more detail the properties of the homological vector 
field \eqref{VarHomVF} if $A\colon\widehat{\varkappa}(\pi)\to\varkappa(\pi)$ is a Hamiltonian differential operator and thus its arguments $\bp$ 
are the \textsl{variational covectors}.\footnote{By definition, the variational covectors are homomorphisms from the space $\varkappa(\pi)=\Gamma\bigl(\pi_\infty^*(\pi)\bigr)$ of generating sections $\vph$ for evolutionary derivations $\cEv_\vph$ to the space of $n$-th horizontal differential forms on $J^\infty(\pi)$ fibred by $\pi_\infty$ over $\varSigma^n$. This implicitly requires to choose the volume form $\Id\boldsymbol{x}$ on the base and thus fix the coupling $\langle\,,\,\rangle$ between the variational covectors and evolutionary vector fields. In these terms, the homological vector field $Q$ for a Hamiltonian operator $A$ is a BV\/-\/field and not just a BRST\/-\/field, see \cite{AKZS}.} 
Under this assumption, 
the induced velocity $\dot{\bp}$ is known (see formula \eqref{EqDogma} below) whenever the evolution $\dot{u}=A(\bp)$ is given. 
The evolutionary vector field \eqref{VarHomVF} on the horizontal jet super\/-\/bundle $\overline{J^\infty}(\Pi\widehat{\pi}_\pi)$ was recently obtained in \cite{JKGolovko2008} for Hamiltonian operators $A$ from exactly this initial standpoint: 
the coefficient of $\dd/\dd\boldb$ was \textsl{postulated} once and forever by formula \eqref{EqDogma} and was not regarded as the bracket $\ib{\boldb}{\boldb}{A}$, whereas the quadratic equation $Q^2=0$ upon the skew\/-\/adjoint operators $A$ then yields variational Poisson structures. 

In the variational Poisson case, the homological vector field \eqref{VarHomVF} can indeed be calculated explicitly
for every Hamiltonian operator $A$. Moreover, this field $Q$ is itself Hamiltonian with respect to the variational Poisson bi\/-\/vector $\bH=\tfrac{1}{2}A(\boldb)\cdot\boldb$ and the canonical symplectic structure \cite{KuperCotangent,DeDonderWeyl} 
on the horizontal infinite jet super\/-\/space $\overline{J^\infty}(\Pi\widehat{\pi}_\pi)$. 
We illustrate this standard algebraic fact by using techniques 
from \cite{Topical}.  

The bracket $\ib{\,}{\,}{A}$ for Hamiltonian operators $A$ emerges
from the Jacobi identity 
for the Lie algebra $\bigl(\bar{H}^n(\pi),\{\,,\,\}_A\bigr)$ 
of the Hamiltonian functionals endowed by $A$ with the Poisson bracket.
Following the notation of the book~\cite{Opava}, we put
$
\ell^{(\bu)}_{A,\psi}(\vph)\mathrel{{:}{=}}\bigl(\cEv_{\vph}(A)\bigr)(\psi)
$ 
for any $\vph\in\varkappa(\pi)$, 
$\psi\in\widehat{\varkappa}(\pi)$, 
and a total differential operator $A\in\CDiff(\widehat{\varkappa}(\pi),
\varkappa(\pi))$.
We note that $\ell^{(\bu)}_{A,\psi}$ is an operator in total derivatives w.r.t.\ its
argument $\vph$ and w.r.t.\ $\psi$ (but not w.r.t.\ the coefficients of $A$),
and hence the adjoint $\bigl(\ell^{(\bu)}_{A,\psi}\bigr)^\dagger$ is well defined.

\begin{lemma}[\cite{Opava}, see a proof in \cite{d3Bous}]\label{AuxLemma}
A 
skew\/-\/adjoint ope\-ra\-tor
$A$ 
is Hamiltonian if and only if the relation 
$\bigl(\cEv_{A(\bp)}(A)\bigr)(\bq)-
  \bigl(\cEv_{A(\bq)}(A)\bigr)(\bp)=
A\bigl(\bigl(\ell_{A,\bp}^{(\bu)}\bigr)^\dagger(\bq)\bigr)$
holds for all $\bp,\bq\in\widehat{\varkappa}(\pi)$. 
Consequently, 
\begin{equation}\label{EqDogma}
\ib{\bp}{\bq}{A}=\left(\ell^{(\bu)}_{A,\bp}\right)^\dagger(\bq),\qquad \bp,\bq\in\widehat{\varkappa}(\pi),
\end{equation}
for a Hamiltonian operator $A$; 
on the other hand, 
this yields a 
convenient criterion under which total differential operators are Hamiltonian.
\end{lemma}

We emphasize that the notation $\ell^{(\bu)}_{A,\bp}$ is \textsl{not}
the same as the linearization $\ell^{(u)}_{A(\bp)}$.
However, if $\bp$ is the fibre coordinate in the horizontal jet bundle 
over $J^\infty(\pi)$,   
then the notation becomes synonymic (e.g., see \cite{JKGolovko2008}). 
%
Moreover, it is readily seen that\footnote{\label{VarS}%
By the convention
$\delta S=\frac{\delta S}{\delta\eta}\delta\eta$, the variations $\delta\eta$ of (possibly, odd) variables $\eta$ are pushed through to the right.
This is rather unfortunate because it anti\/-\/correlates with Dirac's convention of bra-{} covectors standing on left and {}-ket vectors on right in any formula (e.g.\ in every ``bra-c-ket'')
and brings additional signs of a hardly explicable origin into the formulas.}
\[
\tfrac{1}{2}\left(\ell^{(\bu)}_{A,\boldb}\right)^\dagger(\boldb)=
 \tfrac{\delta}{\delta\bu}\bigl(\tfrac{1}{2}A(\boldb)\cdot\boldb\bigr),
\qquad
A(\boldb)=\tfrac{\delta}{\delta\boldb}\bigl(\tfrac{1}{2}A(\boldb)\cdot\boldb\bigr);
\]
the second identity holds because the Hamiltonian operator $A$ is skew\/-\/adjoint. We finally recall from \cite{KuperCotangent,DeDonderWeyl} 
that the odd neighbour $\overline{J^\infty}(\Pi\widehat{\pi}_\pi)\to J^\infty(\pi)$
of the variational cotangent bundle is endowed with the canonical symplectic structure
$\boldsymbol{\omega}=\left(\begin{smallmatrix}\phantom{+}0 & 1\\ -1 & 0\end{smallmatrix}\right)$. Let us remember that the structure $\boldsymbol{\omega}$ is present even if the target space for local
sections of the bundle $\pi$ or, more generally, the target manifold for the mappings $\bu\colon\varSigma^n\to M^m$ is not symplectic. This is in contrast with \cite{AKZS} where a symplectic structure was induced on
the space of mappings ${\EuScript{E}}=\{\varSigma^n\to M^m\}$ from 
the two\/-\/form on $M^m$, i.e., by using the requirement that the even\/-\/dimensional target manifold be symplectic.

Let us summarize the result. Now on the variational level,
it goes in parallel with the properties of the homological vector fields $Q$
within the finite\/-\/dimensional picture 
of \cite{AKZS}.

\begin{state}
For every Hamiltonian differential operator $A$, the homological evolutionary vector field \eqref{VarHomVF} is itself Hamiltonian with respect to the variational Poisson bi\/-\/vector $\bH=\tfrac{1}{2}A(\boldb)\cdot\boldb$ and the canonical symplectic structure $\boldsymbol{\omega}$:
\begin{equation}\label{VarHomVFPoisson}
Q=\cEv^{(\bu)}_{\tfrac{\delta}{\delta\boldb}(\bH)}
+\cEv^{(\boldb)}_{-\tfrac{\delta}{\delta\bu}(\bH)}.
\end{equation}
The Lagrangian $\bL=\dot{\bu}\bp-\bH$, obtained from the Hamiltonian $\bH$ by the Legendre transform, equals $\bL=\tfrac{1}{2}A(\boldb)\cdot\boldb$
and satisfies the variational master equation $\lshad\bL,\bL\rshad=0$,
where $\lshad\,,\,\rshad$ is the variational Schouten bracket.
\end{state}

\begin{example}
Consider the Hamiltonian operator
$A_2=-\tfrac{1}{2}\,\tfrac{\Id^3}{\Id x^3}+
w\,\tfrac{\Id}{\Id x}+\tfrac{\Id}{\Id x}\circ w$
for the KdV equation; the bracket \eqref{EqDogma} is
$\ib{p}{q}{A_2}=p_xq-pq_x$, see \eqref{BrA2}.
The variational Poisson bi\/-\/vector is
$\bH=-\tfrac{1}{4}b_{xxx}b+wb_xb$, whence (see footnote \ref{VarS})
\[
{\delta\bH}/{\delta b}=-\tfrac{1}{2}b_{xxx}+2wb_x+w_xb
\quad \text{and}\quad
{\delta\bH}/{\delta w}=b_xb.
\]
The homological evolutionary vector field \eqref{VarHomVF} equals
$Q=\cEv^{(w)}_{A_2(b)}-\tfrac{1}{2}\cEv^{(b)}_{2b_xb}$, 
which is precisely \eqref{VarHomVFPoisson}.
\end{example}

Let us finally remark that 
the explicitly given\footnote{The Lagrangian $\bL$ 
and the Hamiltonian $\bH$ coincide, 
whence the variational Poisson formalism can be viewed as a 
very remote part of the theory of uniform rectilinear motion.}
Hamiltonian $\bH$ of $Q$ has nothing to do with a Hamiltonian $\cH$ of the actual motion $\dot{\bu}=A(\delta\cH/\delta\bu)$
of the image of $\varSigma^n$ in $M^m$ because 
$\cH$ may not even exist\,!

\paragraph*
{Conclusion.}\noindent%
We observe that the physically motivated geometry of \cite{AKZS} is manifestly \textsl{variational}, that is, the model of strings $\varSigma^n\to M^m$ 
requires us to use jet bundles that encode the information about the mappings together with all their derivatives. It is possible to introduce the constructions such as the 
Poisson structures or homological 
vector fields of zero differential order along $\varSigma^n$,
but, clearly, they neglect the full setup.
Physically speaking, the equality of all derivatives to zero means that either $n=0$ and $\varSigma=\{\text{pt}\}$ (hence there are no derivatives at all) or the image of $\varSigma^n$ under a constant mapping is a point in $M^m$ so that the point particle has $n$ hidden 
dimensions. 
In this paper, we introduced the definitions and established their properties along the lines of \cite{Vaintrob}, and thus we removed the limitation of \cite{AKZS}
for the particles to be only points. Furthermore, we notice that yet one more assumption of \cite{AKZS}, on the presence of a symplectic structure on the target manifold $M^m$, is excessive for the homological evolutionary BV-\/field to be Hamiltonian
with respect to 
the canonical symplectic structure on the space of mappings
(although, of course, variational methods remain applicable to the Poisson sigma\/-\/models). 
We hope that this will help in the study of the dimensional 
reduction\/-\/free models for strings in Minkowski space\/-\/time $M^{3,1}$.


\paragraph*
{Acknowledgements.}
The authors thank 
B.\,A.\,Dub\-ro\-vin, I.\,S.\,Kra\-sil'\-sh\-chik,
M.\,A.\,Ne\-ste\-ren\-ko, P. J. Olver, V. N. Roub\-tsov, 
and V.\,V.\,So\-ko\-lov 
for the discussions and remarks.
One of the authors (A.~K.) is grateful to the organizing committee of the International Workshop ``Nonlinear Physics: Theory and Experiment VI'' for the warm hospitality.
A.\,K.\ thanks the 
$\smash{\mbox{IH\'ES}}$, SISSA, and CRM ($\text{\smash{Montr\'eal}}$) 
for financial support and 
hospitality.

This work was supported in part by the European Union 
(FP6 Marie Curie RTN \emph{ENIGMA} Contract
No.\,MRTN-CT-2004-5652), the European Science Foundation Program
{MISGAM}, and the NWO (Grant Nos.\ B61--609 and VENI 639.031.623).


\begin{thebibliography}{99}

\bibitem{Vaintrob}
\by{Vaintrob A. Yu.} (1997)
Lie algebroids and homological vector fields,
\jour{Russ.\ Math.\ Surv.} \vol{52}:2, 428--429.

\bibitem{AKZS}
\by{Alexandrov M., Schwarz A., Zaboronsky O., Kontsevich M.} (1997)
The geometry of the master equation and topological quantum field
theory, \jour{Int. J.\ Modern Phys.} \vol{A12}:7, 1405--1429.

\bibitem{Voronov}
\by{Voronov T.} (2002) Graded manifolds and Drinfeld doubles for Lie bialgebroids, in: \book{Quantization, Poisson brackets, and beyond}
(Voronov T., ed.) Contemp.\ Math.\ \vol{315}, AMS, Providence, RI, 131--168.

\bibitem{YKSDB}
\by{Kosmann\/--\/Schwarzbach Y.} (2008)
Poisson manifolds, Lie algebroids, modular classes: a survey.
\jour{SIGMA} 
\vol{4}:5, 1--30;
Derived brackets, \jour{Lett.\ Math.\ Phys.} \vol{69} (2004), 61--87;
From Poisson algebras to Gerstenhaber algebras,
\jour{Ann.\ Inst.\ Fourier, Grenoble} \vol{46} (1996), 1243--1274.

\bibitem{ClassSym}
\by{Krasil'shchik I. S., Vinogradov A. M.}, eds.\ (1999)
\book{Symmetries and conservation laws for differential equations of mathematical physics}. (
{Bocharov A. V., Chetverikov V. N., Duzhin S. V.  \textit{et al}.})
AMS, Providence, RI.

\bibitem{Herz}
\by{Herz J.-C.} (1953) Pseudo\/-\/alg\`ebres de Lie. I, II.
\jour{C. R. Acad.\ Sci.\ Paris} \vol{236}, 1935--1937,
2289--2291.

\bibitem{YKSMagri}
\by{Kosmann\/-\/Schwarzbach Y., Magri F.} (1990)
Poisson\/--\/Nijenhuis structures,
\jour{Ann.\ Inst.\ H. Poin\-car\'e\textup{,} ser. A\textup{:} 
Phys.\ Th\'eor.} \vol{53}:1, 35--81.

\bibitem{SymToda}
\by{Kiselev A. V., van de Leur J. W.} (2010) Symmetry algebras of Lagrangian
Liouville\/-\/type systems, \jour{Theoret.\ Math.\ Phys.},
\vol{162}:3, 149--162.\ \texttt{arXiv:nlin.SI/0902.3624}

\bibitem{Leznov}
\by{Leznov A. N., Saveliev M. V.} (1979) Representation of zero curvature for the system of nonlinear partial differential equations $x_{\alpha,z\bar{z}}=\exp(Kx)_\alpha$ and its integrability, \jour{Lett.\ Math.\ Phys.} \vol{3}, 489--494.

\bibitem{SokolovUMN}
\by{Zhiber A.V., Sokolov V.V.} (2001)
  Exactly integrable hyperbolic equations of Liouvillean type,
\jour{Russ.\ Math.\ Surveys} \vol{56}:1, 61--101.

\bibitem{Kumpera}
\by{Kumpera A., Spencer D.} (1972)
\book{Lie equations. I: General theory.}
Annals of Math.\ Stud. \vol{73}. Princeton University Press, Princeton, N.J.;
University of Tokyo Press, Tokyo. 

\bibitem{Topical}
\by{Krasil'shchik J., Verbovetsky A.} (2010)
Geometry of jet spaces and integrable systems,
\jour{J.~Geom.\ Phys.} (to appear), 54~p.\ 
\texttt{arXiv:math.DG/1002.0077} 

\bibitem{KuperCotangent}
\by{Kupershmidt B. A.} (1980)
Geometry of jet bundles and the structure of Lagrangian and
Hamiltonian formalisms. \book{Geometric methods in mathematical
physics}.
Lecture Notes in Math.\ \vol{775}, Springer, Berlin, 162--218.

\bibitem{Olver}
\by{Olver P. J.} (1993) \book{Applications of Lie groups to differential
equations}, Grad.\ Texts in Math.\ \vol{107} (2nd ed.), 
Springer\/--\/Verlag, NY. 

\bibitem{TMPhGallipoli}
\by{Kiselev A.V.} (2005) Hamiltonian flows on Euler-type equations,
\jour{Theoret.\ Math.\ Phys.} \vol{144}:1, 952--960.\
\texttt{arXiv:nlin.SI/0409061}

\bibitem{d3Bous}
\by{Kiselev A. V., van de Leur J. W.} (2009) A family of second Lie algebra
structures for symmetries of dispersionless Boussinesq system,
\jour{J. Phys.\ A\textup{:} Math.\ Theor.}, \vol{42}:40, 
404011 (8 p.) \texttt{arXiv:nlin.SI/0903.1214}

\bibitem{Twente}
\by{Kiselev A. V., van de Leur J. W.} (2007) 
Involutive distributions of operator\/-\/valued 
evolutionary vector fields and their affine geometry,
\jour{Preprint} IHES/M/07/38, 30 p.

\bibitem{Lstar}
\by{Kersten P., Krasil'shchik I., Verbovetsky A.} (2004) Hamiltonian
operators and $\ell^*$-coverings, \jour{J. Geom.\ Phys.} \vol{50}:1--4,
273--302.

\bibitem{Contractions}
\by{Nesterenko M., Popovych R.} (2006) 
Contractions of low\/-\/dimensional Lie
algebras, \jour{J. Math.\ Phys.} \vol{47}:12, 123515, 45 pp.

\bibitem{BoyerFinley}
\by{Boyer C. P., Finley J. D. III} (1982)
Killing vectors in self-dual, Euclidean Einstein spaces,
\jour{J. Math. Phys.} 23:6, 1126--1130.

\bibitem{JKGolovko2008}
\by{Golovko V. A., Krasil'shchik I. S., Verbovetsky A. M.} (2008)
Variational Poisson\/--\/Nijenhuis structures for partial differential equations, \jour{Theoret.\ Math.\ Phys.} \vol{154}:2, 227--239.

\bibitem{Stasheff}
\by{Lada T., Stasheff J.} (1993)
Introduction to \textsc{SH} Lie algebras for physicists,
\jour{Internat.\ J. Theoret.\ Phys.} \vol{32}:7, 1087--1103. 

\bibitem{ForKac}
\by{Kiselev A. V.} (2007) 
Associative homotopy Lie algebras and Wronskians,
\jour{J. Math.\ Sci.} \vol{141}:1, 1016--1030.\ %
\texttt{arXiv:math.RA/0410185}

\bibitem{JKVerb3Vect}
\by{Kersten P., Krasil'shchik I., Verbovetsky A.} (2004) On the
integrability conditions for some structures related to evolution
differential equations, \jour{Acta Appl.\ Math.} \vol{83}:1-2, 167--173.

\bibitem{DeDonderWeyl}
\by{Dirac P. A. M.} (1967)
\book{Lectures on quantum mechanics}. 
Belfer Grad. School of Science Monographs Ser. \vol{2}. 
Acad.\ Press, Inc., NY; 
\by{De Donder Th.} (1935) \book{Th\'eorie invariantive du calcul des variations.} 
Gauthier\/--\/Villars, Paris.

\bibitem{Opava}
\by{Krasil'shchik I., Verbovetsky A.} (1998)
\book{Homological methods in equations of mathematical physics}.
Open Education and Sciences, Opava. \texttt{arXiv:math.DG/9808130}

\end{thebibliography}
\end{document}